\documentclass{amsart}
\usepackage{amssymb}
\usepackage{amsfonts}

\newtheorem{theorem}{Theorem}
\theoremstyle{plain}

\newtheorem{definition}{Definition}

\newtheorem{lemma}{Lemma}

\newtheorem{remark}{Remark}

\numberwithin{equation}{section}
\input{tcilatex}

\begin{document}
\title[Hermite-Hadamard type inequalities via fractional integrals]{
Hermite-Hadamard type inequalities for mappings whose derivatives are $s-$%
convex in the second sense via fractional integrals}
\author{Erhan SET$^{\blacktriangle }$}
\address{$^{\blacktriangle }$Department of Mathematics, \ Faculty of Science
and Arts, D\"{u}zce University, D\"{u}zce-TURKEY}
\email{erhanset@yahoo.com}
\author{M. Emin OZDEMIR$^{\bigstar }$}
\address{$^{\bigstar }$Atat\"{u}rk University, K.K. Education Faculty,
Department of Mathematics, 25240, Campus, Erzurum, Turkey}
\email{emos@atauni.edu.tr}
\author{M. Zeki Sar\i kaya$^{\blacksquare }$}
\address{$^{\blacksquare }$Department of Mathematics, \ Faculty of Science
and Arts, D\"{u}zce University, D\"{u}zce-TURKEY}
\email{sarikayamz@gmail.com}
\author{Filiz KARAKO\c{C}$^{\clubsuit }$}
\address{$^{\clubsuit }$Department of Mathematics, \ Faculty of Science and
Arts, D\"{u}zce University, D\"{u}zce-TURKEY}
\email{filinz\_41@hotmail.com}

\begin{abstract}
In this paper we establish Hermite-Hadamard type inequalities for mappings
whose derivatives are $s-$convex in the second sense and concave.
\end{abstract}

\maketitle

\section{\textbf{Introduction}}

Let f:$I\subseteq 
\mathbb{R}
\rightarrow 
\mathbb{R}
$ be a convex function defined on the interval $I$ of real numbers and $%
a,b\in I$ with $a<b.$ Then

\begin{equation}
f\left( \frac{a+b}{2}\right) \leq \frac{1}{b-a}\dint\limits_{a}^{b}f(x)dx%
\leq \frac{f(a)+f(b)}{2}  \label{(1.1)}
\end{equation}%
is known that the Hermite-Hadamard inequality for convex function. Both
inequalities hold in the reserved direction if $f$ is concave. We note that
Hadamard's inequality may be regarded as a refinement of the concept of
convexity and it follows easily from Jensen's inequality. Hadamard's
inequality for convex functions has received renewed attention in recent
years and a remarkable variety of refinements and generalizations have been
found; see, for example see (\cite{1}-\cite{f6}).

\begin{definition}
(\cite{A4}) A function $f:\left[ 0,\infty \right) \rightarrow 
\mathbb{R}
$ is said to be $s-$convex in the second sense if
\end{definition}

\begin{equation*}
f(\lambda x+(1-\lambda )y)\leq \lambda ^{s}f(x)+(1-\lambda )^{s}f(y)
\end{equation*}%
for all $x,y\in \left[ 0,\infty \right) ,\lambda \in \left[ 0,1\right] $ and
for some fixed $s\in \left( 0,1\right] $. This class of $s-$convex functions
is usually denoted by $K_{s}^{2}$.

In (\cite{A2}) Dragomir and Fitzpatrick proved a variant of Hadamard's
ineqality which holds for $s-$convex functions in the second sense:

\begin{theorem}
\label{11} Suppose that $f:\left[ 0,\infty \right) \rightarrow \left[
0,\infty \right) $ is an $s-$convex function in the second sense ,where $%
s\in \left( 0,1\right) $ and let $a,b\in \left[ 0,\infty \right) ,a<b.$ If $%
f^{\prime }\in L^{1}\left( \left[ a,b\right] \right) ,$ then the following
inequalities hold:%
\begin{equation}
2^{s-1}f\left( \frac{a+b}{2}\right) \leq \frac{1}{b-a}\int_{a}^{b}f(x)dx\leq 
\frac{f(a)+f(b)}{2}  \label{1.2}
\end{equation}%
The constant $k=\frac{1}{s+1}$ is the best possible in the second ineqality
in (\ref{1.2})
\end{theorem}

The following results are proved by M.I.Bhatti et al. (see \cite{A1}).

\begin{theorem}
\label{12} Let $f:I\subseteq 
\mathbb{R}
\rightarrow 
\mathbb{R}
$ be a twice differentiable function on $I^{\circ }$ such that $\left\vert
f^{\prime \prime }\right\vert $ is convex function on $I$. Suppose that $%
a,b\in I^{\circ }$ with $a<b$ and $f^{\prime \prime }\in L\left[ a,b\right] $%
, then the following ineqality for fractional integrals with $\alpha >0$
holds:%
\begin{eqnarray}
&&\left\vert \frac{f(a)+f(b)}{2}-\frac{\Gamma \left( \alpha +1\right) }{%
2\left( b-a\right) ^{\alpha }}\left[ J_{a^{+}}^{\alpha
}f(b)+J_{b^{-}}^{\alpha }f(a)\right] \right\vert  \label{1.3} \\
&\leq &\frac{\left( b-a\right) ^{2}}{\alpha +1}\beta \left( 2,\alpha
+1\right) \left[ \frac{\left\vert f^{\prime \prime }(a)\right\vert
+\left\vert f^{\prime \prime }(b)\right\vert }{2}\right]  \notag
\end{eqnarray}%
where $\beta $ is Euler Beta function.
\end{theorem}

\begin{theorem}
\label{13} Let $f:I\subseteq 
\mathbb{R}
\rightarrow 
\mathbb{R}
$ be a twice differentiable function on $I^{\circ }$. Assume that $p\in 
\mathbb{R}
,p>1$ such that $\left\vert f^{\prime \prime }\right\vert ^{\frac{p}{p-1}}$
is convex function on $I$. Suppose that $a,b\in I^{\circ }$ with $a<b$ and $%
f^{\prime \prime }\in L\left[ a,b\right] ,$ then the following ineqality for
fractional integrals holds:%
\begin{eqnarray}
&&\left\vert \frac{f(a)+f(b)}{2}-\frac{\Gamma \left( \alpha +1\right) }{%
2\left( b-a\right) ^{\alpha }}\left[ J_{a^{+}}^{\alpha
}f(b)+J_{b^{-}}^{\alpha }f(a)\right] \right\vert  \label{1.4} \\
&\leq &\frac{\left( b-a\right) ^{2}}{\alpha +1}\beta ^{\frac{1}{p}}\left(
p+1,\alpha p+1\right) \left( \frac{\left\vert f^{\prime \prime
}(a)\right\vert ^{q}+\left\vert f^{\prime \prime }(b)\right\vert ^{q}}{2}%
\right) ^{\frac{1}{q}}  \notag
\end{eqnarray}%
where $\beta $ is Euler Beta function.
\end{theorem}

\begin{theorem}
\label{14} Let $f:I\subseteq 
\mathbb{R}
\rightarrow 
\mathbb{R}
$ be a twice differentiable function on $I^{\circ }$. Assume that $q\geq 1$
such that $\left\vert f^{\prime \prime }\right\vert ^{q}$ is convex function
on $I.$Suppose that $a,b\in I^{\circ }$ with $a<b$ and $f^{\prime \prime
}\in L\left[ a,b\right] ,$then the following ineqality for fractional
integrals holds:%
\begin{eqnarray}
&&\left\vert \frac{f(a)+f(b)}{2}-\frac{\Gamma \left( \alpha +1\right) }{%
2\left( b-a\right) ^{\alpha }}\left[ J_{a^{+}}^{\alpha
}f(b)+J_{b^{-}}^{\alpha }f(a)\right] \right\vert  \label{1.5} \\
&\leq &\frac{\alpha \left( b-a\right) ^{2}}{4\left( \alpha +1\right) \left(
\alpha +2\right) }\left[ 
\begin{array}{c}
\left( \frac{2\alpha +4}{3\alpha +9}\left\vert f^{\prime \prime
}(a)\right\vert ^{q}+\frac{\alpha +5}{3\alpha +9}\left\vert f^{\prime \prime
}(b)\right\vert ^{q}\right) ^{\frac{1}{q}} \\ 
+\left( \frac{\alpha +5}{3\alpha +9}\left\vert f^{\prime \prime
}(a)\right\vert ^{q}+\frac{2\alpha +4}{3\alpha +9}\left\vert f^{\prime
\prime }(b)\right\vert ^{q}\right) ^{\frac{1}{q}}%
\end{array}%
\right]  \notag
\end{eqnarray}
\end{theorem}

\begin{theorem}
\label{15} Let $f:I\subseteq 
\mathbb{R}
\rightarrow 
\mathbb{R}
$ be a twice differentiable function on $I^{\circ }$. Assume that $p\in 
\mathbb{R}
,p>1$ with $q=\frac{p}{p-1}$ such that $\left\vert f^{\prime \prime
}\right\vert ^{q}$ is concave function on $I.$Suppose that $a,b\in I^{\circ
} $ with $a<b$ and $f^{\prime \prime }\in L\left[ a,b\right] ,$then the
following ineqality for fractional integrals holds:%
\begin{eqnarray}
&&\left\vert \frac{f(a)+f(b)}{2}-\frac{\Gamma \left( \alpha +1\right) }{%
2\left( b-a\right) ^{\alpha }}\left[ J_{a^{+}}^{\alpha
}f(b)+J_{b^{-}}^{\alpha }f(a)\right] \right\vert  \label{1.6} \\
&\leq &\frac{\left( b-a\right) ^{2}}{\alpha +1}\beta ^{\frac{1}{p}}\left(
p+1,\alpha p+1\right) \left\vert f^{\prime \prime }\left( \frac{a+b}{2}%
\right) \right\vert  \notag
\end{eqnarray}%
where $\beta $ is Euler Beta function.
\end{theorem}

We will give some necessary definitions and mathematical preliminaries of
fractional calculus theory which are used further this paper.

\begin{definition}
Let $f\in L\left[ a,b\right] $. The Reimann-Liouville integrals $%
J_{a^{+}}^{\alpha }f(x)$ and $J_{b^{-}}^{\alpha }f(x)$ of order $\alpha >0$
with $\alpha \geq 0$ are defined by
\end{definition}

\begin{equation*}
J_{a^{+}}^{\alpha }f(x)=\frac{1}{\Gamma (\alpha )}\int_{a}^{x}(x-t)^{\alpha
-1}f(t)dt,x>\alpha
\end{equation*}%
and

\begin{equation*}
J_{b^{-}}^{\alpha }f(x)=\frac{1}{\Gamma (\alpha )}\int_{x}^{b}(t-x)^{\alpha
-1}f(t)dt,x<b
\end{equation*}%
respectively. Where $\Gamma (\alpha )=\int_{0}^{\infty }e^{-t}u^{\alpha
-1}du $ is the Gamma function and $J_{a^{+}}^{0}f(x)=J_{b^{-}}^{0}f(x)=f(x).$

In the case of $\alpha =1$ the fractional integral reduces to the classical
integral.

For some recent results connected with fractional integral ineqalities, (see 
\cite{r1}-\cite{r10}).

In this paper, we establish fractional integral ineqalities of
Hermite-Hadamard type for mappings whose derivatives are $s-$convex and
concave.

\section{\textbf{\protect\bigskip Main Results}}

In order to prove our main theorems we need the following lemma (see \cite%
{A1}).

\begin{lemma}
\label{L1} Let $f:I\subseteq 
\mathbb{R}
\rightarrow 
\mathbb{R}
$ be a twice differentiable function on $I^{\circ }$,the interior of $I$.
Assume that $a,b\in I^{\circ }$ with $a<b$ and $f^{^{\prime \prime }}\in L%
\left[ a,b\right] $, then the following identity for fractional integral
with $\alpha >0$ holds:
\end{lemma}

\begin{eqnarray}
&&\frac{f(a)+f(b)}{2}-\frac{\Gamma \left( \alpha +1\right) }{2\left(
b-a\right) ^{\alpha }}\left[ J_{a^{+}}^{\alpha }f(b)+J_{b^{-}}^{\alpha }f(a)%
\right]  \label{2.1} \\
&=&\frac{\left( b-a\right) ^{2}}{2\left( \alpha +1\right) }%
\int_{0}^{1}t\left( 1-t^{\alpha }\right) \left[ f^{\prime \prime }\left(
ta+\left( 1-t\right) b\right) +f^{\prime \prime }\left( \left( 1-t\right)
a+tb\right) \right] dt  \notag
\end{eqnarray}%
where $\Gamma (\alpha )=\int_{0}^{\infty }e^{-t}u^{\alpha -1}du.$

\begin{theorem}
\label{16} Let $f:I\subseteq 
\mathbb{R}
\rightarrow 
\mathbb{R}
$ be a twice differentiable function on $I^{\circ }$ and let $a,b\in
I^{\circ }$ with $a<b$ and $f^{^{\prime \prime }}\in L\left[ a,b\right] $.
If $\left\vert f^{\prime \prime }\right\vert $ is $s-$convex in the second
sense on $I$ for some fixed $s\in \left( 0,1\right] $, then the following
inequality for fractional integrals with $\alpha >0$ holds:
\end{theorem}

\begin{eqnarray}
&&\left\vert \frac{f(a)+f(b)}{2}-\frac{\Gamma \left( \alpha +1\right) }{%
2\left( b-a\right) ^{\alpha }}\left[ J_{a^{+}}^{\alpha
}f(b)+J_{b^{-}}^{\alpha }f(a)\right] \right\vert  \label{2.2} \\
&\leq &\frac{\left( b-a\right) ^{2}}{2\left( \alpha +1\right) }\left[ \frac{%
\alpha }{\left( s+2\right) \left( \alpha +s+2\right) }+\beta \left(
2,s+1\right) -\beta \left( \alpha +2,s+1\right) \right] \left[ \left\vert
f^{\prime \prime }(a)\right\vert +\left\vert f^{\prime \prime
}(b)\right\vert \right]  \notag
\end{eqnarray}%
where $\beta $ is Euler Beta function.

\begin{proof}
From Lemma \ref{L1} since $\left\vert f^{\prime \prime }\right\vert $ is $s-$%
convex in the second sense on $I$, we have

\begin{eqnarray*}
&&\left\vert \frac{f(a)+f(b)}{2}-\frac{\Gamma \left( \alpha +1\right) }{%
2\left( b-a\right) ^{\alpha }}\left[ J_{a^{+}}^{\alpha
}f(b)+J_{b^{-}}^{\alpha }f(a)\right] \right\vert \\
&\leq &\frac{\left( b-a\right) ^{2}}{2\left( \alpha +1\right) }%
\int_{0}^{1}\left\vert t\left( 1-t^{\alpha }\right) \right\vert \left[
\left\vert f^{\prime \prime }\left( ta+\left( 1-t\right) b\right)
\right\vert +\left\vert f^{\prime \prime }\left( \left( 1-t\right)
a+tb\right) \right\vert \right] dt \\
&\leq &\frac{\left( b-a\right) ^{2}}{2\left( \alpha +1\right) }\left[
\int_{0}^{1}t(1-t^{\alpha })\left[ t^{s}\left\vert f^{\prime \prime
}(a)\right\vert +(1-t)^{s}\left\vert f^{\prime \prime }(b)\right\vert \right]
dt+\int_{0}^{1}t(1-t^{\alpha })\left[ (1-t)^{s}\left\vert f^{\prime \prime
}(a)\right\vert +t^{s}\left\vert f^{\prime \prime }(b)\right\vert \right] dt%
\right] \\
&=&\frac{\left( b-a\right) ^{2}}{2\left( \alpha +1\right) }\left[
\int_{0}^{1}t^{s+1}\left( 1-t^{\alpha }\right) dt+\int_{0}^{1}t(1-t^{\alpha
})(1-t)^{s}dt\right] \left[ \left\vert f^{\prime \prime }(a)\right\vert
+\left\vert f^{\prime \prime }(b)\right\vert \right] \\
&=&\frac{\left( b-a\right) ^{2}}{2\left( \alpha +1\right) }\left[ \frac{%
\alpha }{\left( s+2\right) \left( \alpha +s+2\right) }+\beta \left(
2,s+1\right) -\beta \left( \alpha +2,s+1\right) \right] \left[ \left\vert
f^{\prime \prime }(a)\right\vert +\left\vert f^{\prime \prime
}(b)\right\vert \right]
\end{eqnarray*}%
where we used the fact that%
\begin{equation*}
\int_{0}^{1}t^{s+1}\left( 1-t^{\alpha }\right) dt=\frac{\alpha }{\left(
s+2\right) \left( \alpha +s+2\right) }
\end{equation*}
and%
\begin{equation*}
\int_{0}^{1}t(1-t^{\alpha })(1-t)^{s}dt=\beta \left( 2,s+1\right) -\beta
\left( \alpha +2,s+1\right)
\end{equation*}%
which completes the proof.
\end{proof}

\begin{remark}
\label{R1} In Theorem \ref{16} if we choose $s=1$ then (\ref{2.2}) reduces
the ineqality (\ref{1.3}) of Theorem \ref{12}.
\end{remark}

\begin{theorem}
\label{17} Let $f:I\subseteq 
\mathbb{R}
\rightarrow 
\mathbb{R}
$ be a twice differentiable function on $I^{\circ }$. Suppose that $a,b\in
I^{\circ }$ with $a<b$ and $f^{^{\prime \prime }}\in L\left[ a,b\right] $.
If $\left\vert f^{\prime \prime }\right\vert ^{q}$ is $s-$convex in the
second sense on $I$ for some fixed $s\in \left( 0,1\right] $, $p,q>1$ then
the following inequality for fractional integrals holds:
\end{theorem}

\begin{eqnarray}
&&\left\vert \frac{f(a)+f(b)}{2}-\frac{\Gamma \left( \alpha +1\right) }{%
2\left( b-a\right) ^{\alpha }}\left[ J_{a^{+}}^{\alpha
}f(b)+J_{b^{-}}^{\alpha }f(a)\right] \right\vert  \label{2.3} \\
&\leq &\frac{\left( b-a\right) ^{2}}{\alpha +1}\beta ^{\frac{1}{p}}\left(
p+1,\alpha p+1\right) \left[ \frac{\left\vert f^{\prime \prime
}(a)\right\vert ^{q}+\left\vert f^{\prime \prime }(b)\right\vert ^{q}}{s+1}%
\right] ^{\frac{1}{q}}  \notag
\end{eqnarray}

where $\beta $ is Euler Beta function and $\frac{1}{p}+\frac{1}{q}=1$.

\begin{proof}
From Lemma \ref{L1},using the well known H\"{o}lder inequality and $%
\left\vert f^{\prime \prime }\right\vert ^{q}$ is $s-$convex in the second
sense on $I$, we have%
\begin{eqnarray*}
&&\left\vert \frac{f(a)+f(b)}{2}-\frac{\Gamma \left( \alpha +1\right) }{%
2\left( b-a\right) ^{\alpha }}\left[ J_{a^{+}}^{\alpha
}f(b)+J_{b^{-}}^{\alpha }f(a)\right] \right\vert  \\
&\leq &\frac{\left( b-a\right) ^{2}}{2\left( \alpha +1\right) }%
\int_{0}^{1}\left\vert t\left( 1-t^{\alpha }\right) \right\vert \left[
\left\vert f^{\prime \prime }\left( ta+\left( 1-t\right) b\right)
\right\vert +\left\vert f^{\prime \prime }\left( \left( 1-t\right)
a+tb\right) \right\vert \right] dt \\
&\leq &\frac{\left( b-a\right) ^{2}}{2\left( \alpha +1\right) }\left(
\int_{0}^{1}t^{p}\left( 1-t^{\alpha }\right) ^{p}dt\right) ^{1-\frac{1}{q}}%
\left[ \left( \int_{0}^{1}\left\vert f^{\prime \prime
}(ta+(1-t)b)\right\vert ^{q}dt\right) ^{\frac{1}{q}}+\left(
\int_{0}^{1}\left\vert f^{\prime \prime }((1-t)a+tb)\right\vert
^{q}dt\right) ^{\frac{1}{q}}\right]  \\
&\leq &\frac{\left( b-a\right) ^{2}}{2\left( \alpha +1\right) }\left(
\int_{0}^{1}t^{p}\left( 1-t^{\alpha }\right) ^{p}dt\right) ^{1-\frac{1}{q}}%
\left[ 
\begin{array}{c}
\left( \int_{0}^{1}\left( t^{s}\left\vert f^{\prime \prime }(a)\right\vert
^{q}+(1-t\right) ^{s}\left\vert f^{\prime \prime }(b)\right\vert
^{q})dt\right) ^{\frac{1}{q}} \\ 
+\left( \int_{0}^{1}\left( (1-t)^{s}\left\vert f^{\prime \prime
}(a)\right\vert ^{q}+t^{s}\left\vert f^{\prime \prime }(b)\right\vert
^{q}\right) dt\right) ^{\frac{1}{q}}%
\end{array}%
\right]  \\
&=&\frac{\left( b-a\right) ^{2}}{2\left( \alpha +1\right) }\left(
\int_{0}^{1}t^{p}\left( 1-t^{\alpha }\right) ^{p}dt\right) ^{1-\frac{1}{q}}%
\left[ 
\begin{array}{c}
\left( \left\vert f^{\prime \prime }(a)\right\vert ^{q}\frac{1}{s+1}%
+\left\vert f^{\prime \prime }(b)\right\vert ^{q}\frac{1}{s+1}\right) ^{%
\frac{1}{q}} \\ 
+\left( \left\vert f^{\prime \prime }(a)\right\vert ^{q}\frac{1}{s+1}%
+\left\vert f^{\prime \prime }(b)\right\vert ^{q}\frac{1}{s+1}\right) ^{%
\frac{1}{q}}%
\end{array}%
\right]  \\
&\leq &\frac{\left( b-a\right) ^{2}}{\alpha +1}\beta ^{\frac{1}{p}}\left(
p+1,\alpha p+1\right) \left[ \frac{\left\vert f^{\prime \prime
}(a)\right\vert ^{q}+\left\vert f^{\prime \prime }(b)\right\vert ^{q}}{s+1}%
\right] ^{\frac{1}{q}}
\end{eqnarray*}%
where we used the fact that%
\begin{equation*}
\int_{0}^{1}t^{s}dt=\int_{0}^{1}(1-t)^{s}dt=\frac{1}{s+1}
\end{equation*}%
and%
\begin{equation*}
\int_{0}^{1}t^{p}\left( 1-t^{\alpha }\right) ^{p}dt\leq
\int_{0}^{1}t^{p}(1-t)^{\alpha p}dt=\beta ^{\frac{1}{p}}\left( p+1,\alpha
p+1\right) 
\end{equation*}%
which completes the proof.
\end{proof}

\begin{remark}
\label{R2} In Theorem \ref{17} if we choose $s=1$ then (\ref{2.3}) reduces
the ineqality (\ref{1.4}) of Theorem \ref{13}.
\end{remark}

\begin{theorem}
\label{18} Let $f:I\subseteq 
\mathbb{R}
\rightarrow 
\mathbb{R}
$ be a twice differentiable function on $I^{\circ }.$Suppose that $a,b\in
I^{\circ }$ with $a<b$ and $f^{^{\prime \prime }}\in L\left[ a,b\right] $.
If $\left\vert f^{\prime \prime }\right\vert ^{q}$ is $s-$convex in the
second sense on $I$ for some fixed $s\in \left( 0,1\right] $ and $q\geq 1$
then the following inequality for fractional integrals holds:%
\begin{eqnarray}
&&\left\vert \frac{f(a)+f(b)}{2}-\frac{\Gamma \left( \alpha +1\right) }{%
2\left( b-a\right) ^{\alpha }}\left[ J_{a^{+}}^{\alpha
}f(b)+J_{b^{-}}^{\alpha }f(a)\right] \right\vert  \label{2.4} \\
&\leq &\frac{\alpha \left( b-a\right) ^{2}}{4\left( \alpha +1\right) \left(
\alpha +2\right) }  \notag \\
&&\times \left[ 
\begin{array}{c}
\left( \left\vert f^{\prime \prime }(a)\right\vert ^{q}\frac{2\alpha +4}{%
\left( s+2\right) \left( \alpha +s+2\right) }+\left\vert f^{\prime \prime
}(b)\right\vert ^{q}\left[ \beta \left( 2s+1\right) -\beta \left( \alpha
+2,s+1\right) \right] \frac{2\alpha +4}{\alpha }\right) ^{\frac{1}{q}} \\ 
\left( \left\vert f^{\prime \prime }(a)\right\vert ^{q}\left[ \beta \left(
2s+1\right) -\beta \left( \alpha +2,s+1\right) \right] \frac{2\alpha +4}{%
\alpha }+\left\vert f^{\prime \prime }(b)\right\vert ^{q}\frac{2\alpha +4}{%
\left( s+2\right) \left( \alpha +s+2\right) }\right) ^{\frac{1}{q}}%
\end{array}%
\right]  \notag
\end{eqnarray}
\end{theorem}

\begin{proof}
From Lemma \ref{L1}, using power mean ineqality and $\left\vert f^{\prime
\prime }\right\vert ^{q}$ is $s-$convex in the second sense on $I$ we have 
\begin{eqnarray*}
&&\left\vert \frac{f(a)+f(b)}{2}-\frac{\Gamma \left( \alpha +1\right) }{%
2\left( b-a\right) ^{\alpha }}\left[ J_{a^{+}}^{\alpha
}f(b)+J_{b^{-}}^{\alpha }f(a)\right] \right\vert  \\
&\leq &\frac{\left( b-a\right) ^{2}}{2\left( \alpha +1\right) }%
\int_{0}^{1}\left\vert t\left( 1-t^{\alpha }\right) \right\vert \left[
\left\vert f^{\prime \prime }\left( ta+\left( 1-t\right) b\right)
\right\vert +\left\vert f^{\prime \prime }\left( \left( 1-t\right)
a+tb\right) \right\vert \right] dt \\
&\leq &\frac{\left( b-a\right) ^{2}}{2\left( \alpha +1\right) }\left(
\int_{0}^{1}t\left( 1-t^{\alpha }\right) dt\right) ^{1-\frac{1}{q}}\left[ 
\begin{array}{c}
\left( \int_{0}^{1}t(1-t^{\alpha })\left\vert f^{\prime \prime
}(ta+(1-t)b)\right\vert ^{q}dt\right) ^{\frac{1}{q}} \\ 
+\left( \int_{0}^{1}t(1-t^{\alpha })\left\vert f^{\prime \prime
}((1-t)a+tb)\right\vert ^{q}dt\right) ^{\frac{1}{q}}%
\end{array}%
\right]  \\
&\leq &\frac{\left( b-a\right) ^{2}}{2\left( \alpha +1\right) }\left(
\int_{0}^{1}t\left( 1-t^{\alpha }\right) dt\right) ^{1-\frac{1}{q}}\left[ 
\begin{array}{c}
\left( \int_{0}^{1}\left[ t^{s+1}\left( 1-t^{\alpha }\right) \left\vert
f^{\prime \prime }(a)\right\vert ^{q}+t\left( 1-t^{\alpha }\right) \left(
1-t\right) ^{s}\left\vert f^{\prime \prime }(b)\right\vert ^{q}\right]
dt\right) ^{\frac{1}{q}} \\ 
+\left( \int_{0}^{1}t\left( 1-t^{\alpha }\right) \left( 1-t\right)
^{s}\left\vert f^{\prime \prime }(a)\right\vert ^{q}+t^{s+1}\left(
1-t^{\alpha }\right) \left\vert f^{\prime \prime }(b)\right\vert
^{q}dt\right) ^{\frac{1}{q}}%
\end{array}%
\right]  \\
&=&\frac{\left( b-a\right) ^{2}}{2\left( \alpha +1\right) }\left( \frac{%
\alpha }{2\left( \alpha +2\right) }\right) ^{1-\frac{1}{q}} \\
&&\times \left[ 
\begin{array}{c}
\left( \left\vert f^{\prime \prime }(a)\right\vert ^{q}\frac{\alpha }{%
(s+2)(\alpha +s+2)}+\left\vert f^{\prime \prime }(b)\right\vert ^{q}\left[
\beta \left( 2s+1\right) -\beta \left( \alpha +2,s+1\right) \right] \right)
^{\frac{1}{q}} \\ 
+\left( \left\vert f^{\prime \prime }(a)\right\vert ^{q}\left[ \beta \left(
2s+1\right) -\beta \left( \alpha +2,s+1\right) \right] +\left\vert f^{\prime
\prime }(b)\right\vert ^{q}+\left\vert f^{\prime \prime }(b)\right\vert ^{q}%
\frac{\alpha }{(s+2)(\alpha +s+2)}\right) ^{\frac{1}{q}}%
\end{array}%
\right]  \\
&\leq &\frac{\alpha \left( b-a\right) ^{2}}{4\left( \alpha +1\right) \left(
\alpha +2\right) } \\
&&\times \left[ 
\begin{array}{c}
\left( \left\vert f^{\prime \prime }(a)\right\vert ^{q}\frac{2\alpha +4}{%
\left( s+2\right) \left( \alpha +s+2\right) }+\left\vert f^{\prime \prime
}(b)\right\vert ^{q}\left[ \beta \left( 2s+1\right) -\beta \left( \alpha
+2,s+1\right) \right] \frac{2\alpha +4}{\alpha }\right) ^{\frac{1}{q}} \\ 
\left( \left\vert f^{\prime \prime }(a)\right\vert ^{q}\left[ \beta \left(
2s+1\right) -\beta \left( \alpha +2,s+1\right) \right] \frac{2\alpha +4}{%
\alpha }+\left\vert f^{\prime \prime }(b)\right\vert ^{q}\frac{2\alpha +4}{%
\left( s+2\right) \left( \alpha +s+2\right) }\right) ^{\frac{1}{q}}%
\end{array}%
\right] 
\end{eqnarray*}%
where we used the fact that 
\begin{equation*}
\int_{0}^{1}t^{s+1}\left( 1-t^{\alpha }\right) dt=\frac{\alpha }{%
(s+2)(\alpha +s+2)}
\end{equation*}%
and%
\begin{equation*}
\int_{0}^{1}t\left( 1-t^{\alpha }\right) \left( 1-t\right) ^{s}dt=\beta
\left( 2s+1\right) -\beta \left( \alpha +2,s+1\right) 
\end{equation*}%
which completes the proof.
\end{proof}

\begin{remark}
\label{R3} In Theorem (\ref{18}) if we choose $s=1$ then(\ref{2.4}) reduces
the ineqality (\ref{1.5}) of Theorem (\ref{14}).
\end{remark}

The following result holds for $s-$concavity.

\begin{theorem}
\label{19} Let $f:I\subseteq 
\mathbb{R}
\rightarrow 
\mathbb{R}
$ be a twice differentiable function on $I^{\circ }.$Suppose that $a,b\in
I^{\circ }$ with $a<b$ and $f^{^{\prime \prime }}\in L\left[ a,b\right] $.
If $\left\vert f^{\prime \prime }\right\vert ^{q}$ is $s-$concave in the
second sense on $I$ for some fixed $s\in \left( 0,1\right] $ and $p,q>1$
then the following inequality for fractional integrals holds:%
\begin{eqnarray}
&&\left\vert \frac{f(a)+f(b)}{2}-\frac{\Gamma \left( \alpha +1\right) }{%
2\left( b-a\right) ^{\alpha }}\left[ J_{a^{+}}^{\alpha
}f(b)+J_{b^{-}}^{\alpha }f(a)\right] \right\vert  \label{2.5} \\
&\leq &\frac{\left( b-a\right) ^{2}}{\alpha +1}\beta ^{\frac{1}{p}}\left(
p+1,\alpha p+1\right) 2^{\frac{s-1}{q}}\left\vert f^{\prime \prime }\left( 
\frac{a+b}{2}\right) \right\vert  \notag
\end{eqnarray}%
where $\frac{1}{p}+\frac{1}{q}=1$ and $\beta $ is $\beta $ is Euler Beta
function
\end{theorem}

\begin{proof}
From Lemma \ref{L1} and using the H\"{o}lder ineqality we have 
\begin{eqnarray}
&&  \label{2.6} \\
&&\left\vert \frac{f(a)+f(b)}{2}-\frac{\Gamma \left( \alpha +1\right) }{%
2\left( b-a\right) ^{\alpha }}\left[ J_{a^{+}}^{\alpha
}f(b)+J_{b^{-}}^{\alpha }f(a)\right] \right\vert   \notag \\
&\leq &\frac{\left( b-a\right) ^{2}}{2\left( \alpha +1\right) }%
\int_{0}^{1}\left\vert t\left( 1-t^{\alpha }\right) \right\vert \left[
\left\vert f^{\prime \prime }\left( ta+\left( 1-t\right) b\right)
\right\vert +\left\vert f^{\prime \prime }\left( \left( 1-t\right)
a+tb\right) \right\vert \right] dt  \notag \\
&\leq &\frac{\left( b-a\right) ^{2}}{2\left( \alpha +1\right) }\left(
\int_{0}^{1}t^{p}\left( 1-t^{\alpha }\right) ^{p}dt\right) ^{\frac{1}{p}}%
\left[ \left( \int_{0}^{1}\left\vert f^{\prime \prime
}(ta+(1-t)b)\right\vert ^{q}dt\right) ^{\frac{1}{q}}+\left(
\int_{0}^{1}\left\vert f^{\prime \prime }((1-t)a+tb)\right\vert
^{q}dt\right) ^{\frac{1}{q}}\right]   \notag \\
&\leq &\frac{\left( b-a\right) ^{2}}{2\left( \alpha +1\right) }\left(
\int_{0}^{1}t^{p}\left( 1-t^{\alpha }\right) ^{p}dt\right) ^{\frac{1}{p}} 
\notag \\
&&\times \left[ 
\begin{array}{c}
\left( \int_{0}^{1}\left( t^{s}\left\vert f^{\prime \prime }(a)\right\vert
^{q}+(1-t\right) ^{s}\left\vert f^{\prime \prime }(b)\right\vert
^{q})dt\right) ^{\frac{1}{q}} \\ 
+\left( \int_{0}^{1}\left( (1-t)^{s}\left\vert f^{\prime \prime
}(a)\right\vert ^{q}+t^{s}\left\vert f^{\prime \prime }(b)\right\vert
^{q}\right) dt\right) ^{\frac{1}{q}}%
\end{array}%
\right]   \notag
\end{eqnarray}%
Since $\left\vert f^{\prime \prime }\right\vert ^{q}$ is $s-$concave using
ineqality (\ref{1.2}) we get (see \cite{A3})%
\begin{equation}
\int_{0}^{1}\left\vert f^{\prime \prime }(ta+(1-t)b)\right\vert ^{q}dt\leq
2^{s-1}\left\vert f^{\prime \prime }\left( \frac{a+b}{2}\right) \right\vert
^{q}  \label{2.7}
\end{equation}%
and%
\begin{equation}
\int_{0}^{1}\left\vert f^{\prime \prime }((1-t)a+tb)\right\vert ^{q}dt\leq
2^{s-1}\left\vert f^{\prime \prime }\left( \frac{b+a}{2}\right) \right\vert
^{q}  \label{2.8}
\end{equation}%
Using (\ref{2.7}) and (\ref{2.8}) in (\ref{2.6}), we have%
\begin{eqnarray*}
&&\left\vert \frac{f(a)+f(b)}{2}-\frac{\Gamma \left( \alpha +1\right) }{%
2\left( b-a\right) ^{\alpha }}\left[ J_{a^{+}}^{\alpha
}f(b)+J_{b^{-}}^{\alpha }f(a)\right] \right\vert  \\
&\leq &\frac{\left( b-a\right) ^{2}}{\alpha +1}\beta ^{\frac{1}{p}}\left(
p+1,\alpha p+1\right) 2^{\frac{s-1}{q}}\left\vert f^{\prime \prime }\left( 
\frac{a+b}{2}\right) \right\vert 
\end{eqnarray*}%
which completes the proof.
\end{proof}

\begin{remark}
\label{r4} In theorem (\ref{19}) if we choose $s=1$ then (\ref{2.5}) reduces
ineqality (\ref{1.6}) of theorem \ref{15}.
\end{remark}

\end{document}